\documentclass[12pt]{amsart}

\usepackage{amsmath, amssymb, amsthm, mathrsfs, stmaryrd, verbatim, graphicx, wrapfig, lipsum, mathtools, mathabx, epsfig, framed,subfigure,float}
\usepackage[foot]{amsaddr}

\expandafter\def\csname opt@stmaryrd.sty\endcsname{only,shortleftarrow,shortrightarrow}

\usepackage{extpfeil}

\usepackage[hang,flushmargin]{footmisc}

\usepackage[backend=biber, style=alphabetic,backref=true]{biblatex}

\usepackage[bookmarks, bookmarksdepth=2, colorlinks=true, linkcolor=green, citecolor=red, urlcolor=green, bookmarksopen=true, unicode]{hyperref}
\usepackage[noabbrev,capitalize,nameinlink]{cleveref}
\usepackage{xurl}
\hypersetup{breaklinks=true}

\usepackage{mathscinet}

\usepackage{tikz, tikz-cd}
\usetikzlibrary{matrix,arrows,decorations.pathmorphing,cd,topaths,calc,graphs}
\tikzset{>=stealth}
\tikzcdset{arrow style=tikz, diagrams={>=stealth}}
\tikzset{rot90/.style={anchor=south, rotate=90, inner sep=.5mm}}

\usepackage{geometry}
\makeatletter
\def\section{\@startsection{section}{1}%
    \z@{.7\linespacing\@plus\linespacing}{.5\linespacing}%
    {\normalfont\large\scshape\centering}}
\makeatother

\numberwithin{equation}{section}
\newtheorem{theorem}[equation]{Theorem}

\newtheorem{corollary}[equation]{Corollary}
\newtheorem{lemma}[equation]{Lemma}

\newtheoremstyle{named}{}{}{\itshape}{}{\bfseries}{.}{.5em}{\thmname{#3} \thmnumber{#2}}
\theoremstyle{named}

\theoremstyle{definition}

\theoremstyle{remark}
\newtheorem{remark}[equation]{Remark}

\usepackage{xparse}
\NewDocumentCommand{\dotimes}{t_}{\IfBooleanTF{#1}{\otimeop}{\otimes}}
\NewDocumentCommand{\otimeop}{m}{\mathbin{\mathop{\otimes}\displaylimits_{#1}}}
\NewDocumentCommand{\lotimes}{t_}{\IfBooleanTF{#1}{\lotimeop}{\otimes}}
\NewDocumentCommand{\lotimeop}{m}{\mathbin{\mathop{\overset{\mathrm{L}}{\otimes}}\displaylimits_{#1}}}

\NewDocumentCommand{\Rtimes}{t_}{\IfBooleanTF{#1}{\rtimeop}{\times}}
\NewDocumentCommand{\rtimeop}{m}{\mathbin{\mathop{\overset{\mathrm{R}}{\times}}\displaylimits_{#1}}}

\makeatletter
\newcommand{\xmapsfrom}[2][]{%
    \ext@arrow3095\leftarrowfill@{#1}{#2}\mapsfromchar}
\makeatother

\usetikzlibrary{calc,decorations.pathmorphing,shapes}
\newcounter{sarrow}

\newextarrow{\xtwoarrows}{{20}{20}{20}{20}}
{\bigRelbar\bigRelbar{\bigtwoarrowsleft\rightarrow\rightarrow}}

\DeclareMathOperator{\Sym}{Sym}

\DeclareMathOperator{\Spec}{Spec}

\DeclareMathOperator{\HH}{HH}  
\DeclareMathOperator{\HN}{HN}  
\DeclareMathOperator{\dR}{dR}  

\makeatletter
\newcommand{\colim@}[2]{%
  \vtop{\m@th\ialign{##\cr
    \hfil$#1\operator@font colim$\hfil\cr
    \noalign{\nointerlineskip\kern1.5\ex@}#2\cr
    \noalign{\nointerlineskip\kern-\ex@}\cr}}%
}
\newcommand{\colim}{%
  \mathop{\mathpalette\colim@{\rightarrowfill@\textstyle}}\nmlimits@
}
\renewcommand{\varprojlim}{%
  \mathop{\mathpalette\varlim@{\leftarrowfill@\scriptscriptstyle}}\nmlimits@
}
\renewcommand{\varinjlim}{%
  \mathop{\mathpalette\varlim@{\rightarrowfill@\scriptscriptstyle}}\nmlimits@
}
\makeatother

\newcommand{\BC}{\mathbb{C}}

\newcommand{\BH}{\mathbb{H}}

\newcommand{\BL}{\mathbb{L}}

\newcommand{\BP}{\mathbb{P}}

\newcommand{\cA}{\mathcal{A}}

\newcommand{\bF}{\mathbf{F}}
\newcommand{\cF}{\mathcal{F}}

\newcommand{\cG}{\mathcal{G}}

\newcommand{\cI}{\mathcal{I}}

\newcommand{\cK}{\mathcal{K}}

\newcommand{\cO}{\mathcal{O}}

\newcommand{\rR}{\mathrm{R}}

\newcommand{\cT}{\mathcal{T}}

\newcommand{\BLb}{\mathbb{L}^{\bullet}}

\title{Hodge to de Rham degeneration of nodal curves}
\author{Yunfan He}
\email{he\_yunfan@pku.edu.cn}

\begin{document}
\begin{abstract}
We compute the Hochschild and negative cyclic homology of the nodal curves, and we show that the (noncommutative) Hodge to de Rham spectral sequence degenerates at the second page. We also classify all the Hochschild classes that can be lifted to negative cyclic homology. In particular, the result in the nodal cubic curve case is important for computation of categorical enumerative invariants.
\end{abstract}
\maketitle

\section{Introduction}
The study of algebraic de Rham cohomology originates with Grothendieck. In \cite{Grothendieck}, he proved that for a smooth scheme \( X \) over \(\BC\), the hypercohomology
\[
H^{\bullet}_{\mathrm{dR}}(X) := \BH^{\bullet}(X, \Omega_X^{\bullet})
\]
of the complex of sheaves of differentials \(\Omega_X^{\bullet}\) computes the singular cohomology of the analytification of \( X \). This complex, now known as the \textit{algebraic de Rham complex}, admits the trivial filtration, which induces a spectral sequence
\[
\prescript{1}{}E^{p,q} = H^p(X, \Omega_X^q) \Longrightarrow H^{p+q}_{\mathrm{dR}}(X),
\]
called the \textit{Hodge-to-de Rham (HdR) spectral sequence}. Grothendieck further showed that if \( X \) is smooth over \(\BC\), this spectral sequence degenerates at the first page. Deligne and Illusie \cite{DeligneIllusie87} later generalized this result to smooth proper schemes over any field \( k \) of characteristic zero via reduction to positive characteristic.

The groups appearing in the HdR spectral sequence are linked to Hochschild homology and negative cyclic homology, which allows generalizing the HdR spectral sequence to noncommutative geometry. The resulting spectral sequence  
\[
\HH_{\bullet}(X)[[u]] \Longrightarrow \HN_{\bullet}(X),
\]  
called the \textit{Hochschild-to-cyclic spectral sequence}, further extends to any \(\cA_{\infty}\)-algebra \(A\). Kontsevich and Soibelman \cite{KontsevichSoibelman09} conjectured that for any smooth and proper \(\cA_{\infty}\)-algebra \(A\) over a field of characteristic \(0\), this spectral sequence degenerates at the first page. Kaledin \cite{Kaledin08,Kaledin17} later proved this conjecture.  

In this paper, we study the case when \(X\) is singular. Specifically, we focus on projective curves of genus \(g\) with \(n\) nodal singularities, which form boundary components of the moduli space \(\overline{M}_{g,n}\). Our main theorem is 
\begin{theorem}
The Hodge to de Rham spectral sequence of any projective curve with nodal singularities degenerates at $\prescript{2}{}E$.
\end{theorem}

In the smooth projective case, the Hochschild-to-cyclic spectral sequence degenerates at $\prescript{1}{}E$, implying that every Hochschild homology class lifts to negative cyclic homology; equivalently, the map $\HN_n \to \HH_n$ is surjective for all $n$. This property fails for singular $X$. However, by explicitly computing $\HH_*(X)$ and $\HN_*(X)$, we prove that the Hochschild-to-cyclic spectral sequence for nodal curves degenerates at $\prescript{2}{}E$. This allows us to classify liftable Hochschild classes and analyze the map $\HN_*(X) \to \HH_*(X)$.  

The liftability of Hochschild classes, particularly the class in $\HH_{-1}(X)$ for a projective nodal cubic curve $X$, plays a key role in computing categorical enumerative invariants \cite{Caldararu-Tu_24}. We elaborate on this in Section \ref{Liftable classes}.  

Sections \ref{Hochschild homology} through \ref{Negative cyclic homology} focus on the case $(g,n) = (1,1)$, corresponding to a projective nodal cubic curve $X$. Section \ref{Hochschild homology} computes the Hochschild homology $\HH_*(X)$, while Section \ref{Hodge to de Rham spectral sequence} examines the degeneration of the HdR spectral sequence for $X$. Section \ref{Negative cyclic homology} determines the negative cyclic homology $\HN_*(X)$. In Section \ref{General cases}, we generalize these computations to arbitrary genus $g$ curves $Y$ with $n$ singularities. Section \ref{Liftable classes} characterizes all liftable Hochschild classes of $X$. Finally, Section \ref{Appendix: cuspidal curve} presents analogous results for cuspidal curves.

\subsection*{Acknowledgements} We would like to thank Andrei C\u{a}ld\u{a}raru, Benjamin Antieau, and Dima Arinkin for helping out at various stages of this project. The author was partially supported by the National Science Foundation through grants number DMS-2152088 and DMS-2202365.

\section{Hochschild homology}\label{Hochschild homology}
Let $X\subset \BP^2$ be the nodal cubic curve. Explicitly, $X$ is cut out by the equation $y^2z=x^3-x^2z$. In this section we will compute $\HH_*(X)$.

Let $L_X$ denote the cotangent sheaf of $X$. 

\begin{lemma}
There are two descriptions of \( L_X \).
\begin{enumerate}
    \item \( L_X \) admits a resolution 
    \[
    0 \to \cO_X(-3) \to \Omega_{\mathbb{P}^2}^1\vert_X \to L_X \to 0,
    \]
    \item \( L_X \) fits into a nontrivial extension 
    \[
    0 \to \cO_P \to L_X \to \cI_P \to 0,
    \]
    where \( P = [0:0:1] \) is the node of \( X \), and \( \cI_P \) is the ideal sheaf of \( P \).
\end{enumerate}
\end{lemma}

\begin{proof}
The first statement follows from the fact that \( X \subset \mathbb{P}^2 \) is a local complete intersection, with conormal sheaf  
\[
\cI_X/\cI_X^2 \cong \cO_{\mathbb{P}^2}(-3)\vert_X \cong \cO_X(-3).
\]

For the second part, consider a coherent sheaf \( \cF \). Every such sheaf fits into a short exact sequence
\[
0 \to \cT \to \cF \to \cG \to 0,
\]
where \( \cT \) is the torsion subsheaf of \( \cF \) and \( \cG \) is torsion-free. Applying this to \( \cF = L_X \), a local calculation shows that \( \cT \cong \cO_P \) and \( \cG \) is locally isomorphic to \( \cI_P \). Thus, \( \cG \cong \cI_P \otimes \cK \) for some line bundle \( \cK \) on \( X \). Using part (1), we compute the Euler characteristic \( \chi(X, L_X) = 0 \), which implies \( \cK \cong \cO_X \). Hence, \( \cG \cong \cI_P \), and the sequence becomes
\[
0 \to \cO_P \to L_X \to \cI_P \to 0. \qedhere
\]
\end{proof}

\begin{remark}
We call the two-term complex \( \cO_X(-3) \to \Omega_{\mathbb{P}^2}^1 \vert_X \), concentrated in degrees \([-1, 0]\), the \textit{cotangent complex} of \( X \) and denote it by \( \BLb_X \). Since \( X \) is a complete intersection, this description agrees with Illusie’s definition of the cotangent complex.
\end{remark}

With these two descriptions of \( L_X \), we obtain:

\begin{corollary}
The cohomology of \( L_X \) is given by
\[
H^i(X, L_X) = \begin{cases} 
\mathbb{C} & i = 0, 1, \\
0 & \text{otherwise}.
\end{cases}
\]
\end{corollary}

\begin{proof}
We have already computed \( \chi(X, L_X) = 0 \). Combining this with the long exact sequence in cohomology induced by the second description of \( L_X \), the result follows.
\end{proof}

\begin{lemma}\label{wedge2}
The hypercohomology of \(\bigwedge^2 \BLb_X\) is given by
\[
\mathbb{H}^i(X, \bigwedge^2 \BLb_X) = \begin{cases} 
\mathbb{C} & i = 0, -1, \\
0 & \text{otherwise}.
\end{cases}
\]
(Here, \(\bigwedge^2\) denotes the derived exterior product.)
\end{lemma}

\begin{proof}
A direct calculation shows that
\[
\bigwedge^2 \BLb_X \coloneqq \Sym^2(\BLb_X[1])[-2] \simeq 
\left(0 \to \cO_X(-6) \to \Omega_{\mathbb{P}^2}^1\vert_X(-3) \to \cO_X(-3) \to 0\right)[0],
\]
which further simplifies to
\[
\cO_X(-3) \otimes \left(0 \to \cO_X(-3) \to \Omega_{\mathbb{P}^2}^1\vert_X \to \cO_X \to 0\right)[0].
\]

Restricting to the affine open subset \(\mathcal{D}(z) \cap X\), the short exact sequence 
\[
0 \to \cO_X(-3) \to \Omega_{\mathbb{P}^2}^1\vert_X \to \cO_X \to 0
\]
becomes
\begin{equation*}\label{ses}
\begin{tikzcd}[row sep=small]
0 \ar[r] & R \ar[r, hook] & R^{\oplus 2} \ar[r, two heads] & R \ar[r] & 0 \\
& 1 \ar[r, mapsto] & \begin{pmatrix} 3x^2 + 2x \\ -2y \end{pmatrix}, \\
&& \begin{pmatrix} f \\ g \end{pmatrix} \ar[r, mapsto] & 2yf + (3x^2 + 2x)g,
\end{tikzcd}
\end{equation*}
where \( R = \mathbb{C}[x, y]/(x^3 + x^2 - y^2) \). Direct computation reveals:
\begin{itemize}
    \item \( H^0 = \mathbb{C} \), generated by \( 1 \),
    \item \( H^{-1} = \mathbb{C} \), generated by \( \begin{pmatrix} (3x + 2)y \\ 2(x^2 + x) \end{pmatrix} \).
\end{itemize}
The generator of \( H^{-1} \) is annihilated by \( x \) and \( y \), implying the \( -1 \)-cohomology sheaf is supported at the node \( P \). Thus, the sequence is quasi-isomorphic to 
\[
0 \to 0 \to \cO_P \to \cO_P \to 0.
\]
Twisting by \( \cO_X(-3) \) preserves hypercohomology, yielding the claimed result.
\end{proof}

With Lemma \ref{wedge2}, the hypercohomology of higher exterior powers of \(\BLb_X\) can be computed analogously.

\begin{corollary}\label{hypercohomology}
The hypercohomology of \(\bigwedge^k \BLb_X\) is given by
\[
\mathbb{H}^i\left(X, \bigwedge^k \BLb_X\right) = \begin{cases} 
\mathbb{C} & i = -k + 2, -k + 1, \\
0 & \text{otherwise}.
\end{cases}
\]
\end{corollary}

\begin{proof}
A direct calculation yields:
\begin{align*}
\bigwedge^k \BLb_X &\simeq \left(0 \to \cO_X(-3k) \to \Omega_{\mathbb{P}^2}^1\vert_X(-3k + 3) \to \cO_X(-3k + 3) \to 0\right)[k - 2] \\
&\simeq \cO_X(-3k + 3) \otimes \left(0 \to \cO_X(-3) \to \Omega_{\mathbb{P}^2}^1\vert_X \to \cO_X \to 0\right)[k - 2] \\
&\simeq \cO_X(-3k + 3) \otimes \left(0 \to 0 \to \cO_P \to \cO_P \to 0\right)[k - 2].
\end{align*}
The result follows by repeating the argument from Lemma \ref{wedge2}.
\end{proof}

Combining these calculations, we obtain:

\begin{theorem}
The Hochschild homology of the projective nodal cubic curve \(X\) is given by
\[
\HH_n(X) = \begin{cases} 
\mathbb{C} & \text{if } n = -1, \\
\mathbb{C}^2 & \text{if } n = 0, \\
\mathbb{C} & \text{if } n \geq 1, \\
0 & \text{otherwise}.
\end{cases}
\]
\end{theorem}

\begin{proof}
The proof relies on a decomposition of the Hochschild complex from \cite{Buchweitz-Flenner}:
\[
\mathsf{HH}_{X/Y} \simeq \Sym^\bullet(\BLb_{X/Y}[1]) \cong \bigoplus_{n} \left(\bigwedge^n \BLb_{X/Y}\right)[n].
\]
For \(\mathsf{HH}_X := \mathsf{HH}_{X/k} = \Sym^\bullet(\BLb_X[1])\), taking hypercohomology gives
\[
\HH_{-*}(X) := \mathbb{R}^{*}\Gamma\left(\Sym^\bullet_{\cO_X}(\BLb_X[1])\right).
\]
The result follows from the preceding calculations and the fact that \(H^0(X, \cO_X) = H^1(X, \cO_X) = \mathbb{C}\).
\end{proof}

\section{Hodge to de Rham spectral sequence}\label{Hodge to de Rham spectral sequence}
Recall that for a smooth projective variety \( W \) of dimension \( n \), the \textit{de Rham complex} is the complex of sheaves  
\[
0 \to \cO_W \to \Omega_W^1 \to \Omega_W^2 \to \cdots \to \Omega_W^n \to 0,
\]
where \( \Omega_W^i \) denotes the sheaf of Kähler differential \( i \)-forms on \( W \).  

For a singular variety \( X \) (such as the nodal cubic curve), the analogous object is the \textit{derived de Rham complex}:
\[
\widehat{\dR}^{\bullet} : 0 \to \cO_X \to \BLb_X \to \bigwedge^2 \BLb_X \to \bigwedge^3 \BLb_X \to \cdots.
\]
Unlike the smooth case, this complex is generally unbounded for singular varieties. However, no completion is required for the nodal cubic curve \( X \), as \( \BLb_X \) is bounded.  

The Hodge filtration on \( \widehat{\dR}^{\bullet} \) induces a spectral sequence with \( E_1 \)-page
\[
\prescript{1}{}E^{p,q} = \mathbb{R}^{p+q}\Gamma\left(\widehat{\dR}^p[-p]\right) = \mathbb{R}^q\Gamma\left(\widehat{\dR}^p\right) = H^q\left(X, \bigwedge^p \BLb_X\right),
\]
which converges to \( H^{p+q}_{\mathrm{sing}}(X, \mathbb{C}) \) (see \cite{Bhatt}). We refer to this as the \textit{Hodge-to-de Rham spectral sequence}. Our main theorem is the following:

\begin{theorem}\label{HdR}
The Hodge to de Rham spectral sequence for $X$ degenerates at page $\prescript{2}{}{E}$.
\end{theorem} 

Before the proof, we write down a few terms in the first page, given explicitly as
\[\begin{tikzcd}[row sep=tiny,column sep=small]
1 & H^1(\cO_X)\ar[r,"\sigma"] & H^1(L_X)\\
0 & H^0(\cO_X)\ar[r,"\alpha"] & H^0(L_X)\ar[r,"\gamma"] & H^0(\bigwedge^2\BLb_X) \\
-1 & & & H^{-1}(\bigwedge^2 \BLb_X)\ar[r,"\beta_1"] & H^{-1}(\bigwedge^3\BLb_X)\\
-2 & & & & H^{-2}(\bigwedge^3\BLb_X)\ar[r,"\beta_2"] & H^{-2}(\bigwedge^4\BLb_X)\\
\vdots& & & & & \ddots & \ddots
\end{tikzcd} \]
Notice that from our calculation in last section, all these terms are $1$-dimensional.

We outline our proof in four steps, established in a series of lemmas.
\begin{itemize}
\item Step 1: $\alpha=0$, Lemma \ref{alpha}.
\item Step 2: $\beta_k$ is an isomorphism for all $k\geq 1$, Lemma \ref{beta}.
\item Step 3: $\gamma$ is an isomorphism, Lemma \ref{gamma}.
\item Step 4: $\sigma=0$, Lemma \ref{sigma}.
\end{itemize}

\begin{lemma}\label{alpha}
The map \(\alpha: H^0(\cO_X) \to H^0(L_X)\) is zero.
\end{lemma}

\begin{proof}
The Hodge-to-de Rham spectral sequence for \(X\) converges to \(H^0_{\mathrm{sing}}(X, \mathbb{C}) \cong \mathbb{C}\). In the \(\prescript{1}{}E\)-page, the term \(H^0(\cO_X)\) lies on the \(0\)-diagonal (where \(p + q = 0\)) and is the only nontrivial term in this diagonal. For the spectral sequence to converge to a one-dimensional space \(H^0_{\mathrm{sing}}(X, \mathbb{C})\) on the 0-th diagonal, no nontrivial differentials can act on \(H^0(\cO_X)\). Thus, \(H^0(\cO_X)\) must survive to the \(\prescript{\infty}{}E\)-page, implying \(\alpha = 0\).  
\end{proof}

\begin{lemma}\label{beta}
The maps $\beta_k:H^{-k}(\bigwedge^{k+1}\BLb_X)\to H^{-k}(\bigwedge^{k+2}\BLb_X)$ are all isomorphisms for $k\geq 1$.
\end{lemma}

\begin{proof}
First observation is that $\bigwedge^{k+1} \BLb_X$ is supported at the node when $k\geq 1$, so the computation of the maps $\beta_k$ is local. Thus we can use a local affine model around the node, i.e., the affine nodal curve.

Let $V=\BC\langle e_x,e_y,e_{\epsilon}\rangle$ be a graded $\BC$-vector space with $\deg e_x=\deg e_y=0$ and $\deg e_{\epsilon}=-1$. Then $S:=\Sym^*V\cong \BC[x,y]\otimes \BC[\epsilon]/(\epsilon^2)$ forms a graded $\BC$-algebra with $\deg x=\deg y=0, \deg \epsilon=-1$. When endowing it with a $\BC[x,y]$-linear differential $\delta$ on $S$ as follows:
\[\BC[x,y]\cdot \epsilon \xrightarrow{\delta:\epsilon \mapsto x^3+x^2-y^2} \BC[x,y], \]
$(S,\delta)$ forms a differential graded algebra that is quasi-isomorphic to $R=\BC[x,y]/(x^3+x^2-y^2)$, i.e., $S$ is a differential graded resolution of $R$. Thus to compute $\bigwedge^k\BLb_X$ when $k\geq 2$, we can use the dg model $\bigwedge^kL_{S/\BC}$. 

Recall that $L_{S/\BC}\cong S\otimes_{\BC} V$ is a $\BC[x,y]$-module. We claim that there is a way to construct a $\BC[x,y]$-linear differential $\Delta$ such that $(L_{S/k},\Delta)$ forms a differential graded $S$-module. Explicitly, $(L_{S/\BC},\Delta)$ is defined as follows:

\[\begin{tikzcd}
\BC[x,y]\cdot\epsilon\otimes d\epsilon\ar[r,"\Delta_2"] & \begin{matrix}
\BC[x,y]\cdot 1\otimes d\epsilon \\
\oplus \\
\BC[x,y]\cdot \epsilon \otimes dx\\
\oplus \\
\BC[x,y]\cdot \epsilon \otimes dy
\end{matrix}\ar[r,"\Delta_1"] & \begin{matrix}
\BC[x,y]\cdot 1\otimes dx \\
\oplus \\
\BC[x,y]\cdot 1\otimes dy
\end{matrix}
\end{tikzcd} \]
Here we use $dx, dy ,d\epsilon$ to denote $e_x,e_y,e_{\epsilon}$ respectively. The differentials $\Delta_1,\Delta_2$ are defined as:
\[\Delta_2(\epsilon\otimes d\epsilon):=(x^3+x^2-y^2)\cdot1\otimes d\epsilon -(3x^2+2x)\cdot\epsilon\otimes dx +2y\cdot \epsilon \otimes dy, \]
\[\Delta_1(1\otimes d\epsilon):=(3x^2+2x)\cdot1\otimes dx -2y\cdot 1\otimes dy, \]
\[\Delta_1(\epsilon \otimes dx):=\epsilon \cdot 1\otimes dx, \]
\[\Delta_1(\epsilon \otimes dy):=\epsilon \cdot 1\otimes dy. \]
With the above definitions, it is easy to check that $(L_{S/\BC},\Delta)$ is a differential graded $(S,\delta)$-module, and there exists a naturally defined chain map $d^{dR}:S\to L_{S/\BC}$ called the de Rham differential as follows:
\[\begin{tikzcd}
& \BC[x,y]\cdot\ar[r,"\delta"]\ar[d,"d^{dR}_1"] \epsilon & \BC[x,y] \ar[d,"d^{dR}_0"] \\
\BC[x,y]\cdot\epsilon\otimes d\epsilon\ar[r,"\Delta_2"] & \begin{matrix}
\BC[x,y]\cdot 1\otimes d\epsilon \\
\oplus \\
\BC[x,y]\cdot \epsilon \otimes dx\\
\oplus \\
\BC[x,y]\cdot \epsilon \otimes dy
\end{matrix}\ar[r,"\Delta_1"] & \begin{matrix}
\BC[x,y]\cdot 1\otimes dx \\
\oplus \\
\BC[x,y]\cdot 1\otimes dy
\end{matrix}
\end{tikzcd} \]
Here, 
\[d^{dR}_0(f):=\frac{\partial f}{\partial x}\cdot1\otimes dx+\frac{\partial f}{\partial y}\cdot 1\otimes dy,\] 
and 
\[d^{dR}_1(g\cdot \epsilon):=g\cdot 1\otimes d\epsilon +\frac{\partial g}{\partial x}\cdot \epsilon \otimes dx+\frac{\partial g}{\partial y}\cdot \epsilon \otimes dy. \]
With the above definition of $(L_{S/\BC},\Delta)$ one can compute higher exterior powers $\bigwedge^k L_{S/\BC}$, and they will still be dg $S$-modules. One can extend the definition of $d^{dR}$ to higher exterior powers to form the de Rham complex $(\bigwedge^{\bullet}L_{S/\BC},d^{dR})$. 

For our purpose, it will be enough to show that the map $H^{-k}(\bigwedge^{k+1}L_{S/\BC})\to H^{-k}(\bigwedge^{k+2}L_{S/\BC})$ are surjective. However, it is actually enough to prove this just for $k=1$, since $\bigwedge^{k+1}L_{S/\BC}$ is basically a shift of $\bigwedge^2L_{S/\BC}$ when $k\geq 1$. 

When $k=1$ we only need to consider the degree -1 part of the morphism $d^{dR}:\bigwedge^2L_{S/\BC}\to \bigwedge^3L_{S/\BC}$. By an explicit calculation we have
\[\begin{tikzcd}
\begin{matrix}
\BC[x,y]\cdot 1\otimes (dx\otimes d\epsilon) \\
\oplus \\
\BC[x,y]\cdot 1 \otimes (dy\otimes d\epsilon)\\
\oplus \\
\BC[x,y]\cdot \epsilon \otimes (dx\wedge dy)
\end{matrix}\ar[r,"\Delta^{\wedge2}_1"]\ar[d,"d^{dR}"] & \BC[x,y]\cdot 1\otimes (dx\wedge dy)\\
\BC[x,y]\cdot 1\otimes(dx\wedge dy)\otimes d\epsilon
\end{tikzcd} \]
Here we use $\Delta^{\wedge 2}$ to denote the differential in $\bigwedge^2L_{S/\BC}$. In matrix form it can be written as 
$\Delta^{\wedge2}_1=\begin{pmatrix}
-2y & -3x^2-2x & x^3+x^2-y^2
\end{pmatrix}$. The de Rham differential is given as 
\[d^{dR} : \begin{pmatrix}
f\\ g\\ h
\end{pmatrix} \mapsto \left(\frac{\partial g}{\partial x}-\frac{\partial f}{\partial y}+h\right)\cdot 1\otimes (dx\wedge dy)\otimes d\epsilon. \]
In particular, $\begin{pmatrix}
-3xy-2y\\2x^2+2x\\ 6x+4
\end{pmatrix}\in \ker (\Delta^{\wedge 2}_1)$ and $d^{dR}(\begin{pmatrix}
-3xy-2y\\2x^2+2x\\ 6x+4
\end{pmatrix})=(13x+8)\cdot 1\otimes (dx\wedge dy)\otimes d\epsilon$ is a generator for $H^{-1}(\bigwedge^3L_{S/\BC})\cong \BC \cdot 1\otimes (dx\wedge dy)\otimes d\epsilon$. Thus $H^{-1}(\bigwedge^2 L_{S/\BC}) \to H^{-1}(\bigwedge^3 L_{S/\BC})$ is surjective, and the same holds for $H^{-1}(\bigwedge^2\BLb_X)\to H^{-1}(\bigwedge^3 \BLb_X)$.
\end{proof}

\begin{lemma}\label{gamma}
The map \(\gamma: H^0(L_X) \to H^0\left(\bigwedge^2\BLb_X\right)\) is an isomorphism.
\end{lemma}

\begin{proof}
It suffices to prove surjectivity of \(\gamma\). From Lemma \ref{beta}, \(L_X\) fits into the short exact sequence
\[
0 \to \cO_P \to L_X \to \cI_P \to 0,
\]
where \(P\) is the node. Since \(\dim H^0(L_X) = \dim H^0(\cO_P) = 1\), we work with the commutative diagram:
\[
\begin{tikzcd}[column sep=small]
H^0(\cO_P)\ar[r]\ar[d,equal] & H^0(L_X)\ar[r,"\gamma"]\ar[d] & H^0\left(\bigwedge^2\BLb_X\right)\ar[d,"\simeq"'{anchor=south,rotate=-90}]\\
H^0(\cO_P)\ar[r]\ar[d,equal] & H^0(L_X\vert_{U})\ar[r]\ar[d,dash,"\simeq"'{anchor=south,rotate=-90}] & H^0\left(\bigwedge^2\BLb_X\vert_{U}\right)\ar[d,dash,"\simeq"'{anchor=south,rotate=-90}]\\
H^0(\cO_P)\ar[r]\ar[d,equal] & H^0(L_Z\vert_{V})\ar[r] & H^0\left(\bigwedge^2\BLb_Z\vert_{V}\right)\\
H^0(\cO_P)\ar[r] & H^0(L_Z)\ar[r]\ar[u] & H^0\left(\bigwedge^2\BLb_Z\right)\ar[u,"\simeq"'{anchor=south,rotate=-90}]
\end{tikzcd}
\]
Here, \(Z = \Spec \BC[x,y]/(xy)\) is the local affine nodal scheme, and \(U, V\) are formal neighborhoods of \(P\) in \(X\) and \(Z\), respectively. The vertical isomorphisms hold because \(\bigwedge^2\BLb_X\) is supported at \(P\). The key identification \(H^0(L_X\vert_U) \simeq H^0(L_Z\vert_V)\) follows from $L_X\vert_{U}\simeq L_{U}\simeq L_{V}\simeq L_Z\vert_{V}$ \cite[Proposition 3.9]{Perez}.

For \(Z\), the cotangent complex \(L_Z\) admits the resolution:
\[
0 \to T \xrightarrow{1 \mapsto x\,dy + y\,dx} T \cdot dx \oplus T \cdot dy \to L_Z \to 0,
\]
where \(T = \BC[x,y]/(xy)\). The induced maps on global sections give:
\[
\begin{tikzcd}[row sep=small]
T/(x,y) \ar[r] & \frac{T \cdot dx \oplus T \cdot dy}{x\,dy + y\,dx} \ar[r] & T/(x,y) \cdot dx \wedge dy \\
1 \ar[r,mapsto] & x\,dy \ar[r,mapsto] & 1 \cdot dx \wedge dy
\end{tikzcd}
\]
This shows \(H^0(\cO_P) \to H^0(L_Z) \to H^0\left(\bigwedge^2\BLb_Z\right)\) is an isomorphism. By formal local equivalence \(X \simeq Z\) at \(P\), \(\gamma\) is surjective (hence an isomorphism) for \(X\). 
\end{proof}

\begin{lemma}\label{sigma}
The map $\sigma:H^1(\cO_X)\to H^1(L_X)$ is $0$.
\end{lemma}

\begin{proof}
The key idea, due to Benjamin Antieau, is to compare the spectral sequence to the spectral sequence associated to the normalization $\widetilde{X}$ of $X$. Explicitly, consider the resolution of singularities $\pi: \widetilde{X}\to X$, where $\widetilde{X}$ is isomorphic to $\BP^1$. Then consider the Hodge to de Rham spectral sequence associated to $\widetilde{X}$, where the first page is given by $\prescript{1}{}{E}_{\widetilde{X}}^{p,q}:=H^q(\widetilde{X},\Omega_{\widetilde{X}}^p)$. The morphism \(\pi\) induces a map of spectral sequences $E_X\to E_{\widetilde{X}}$, yielding a commutative diagram: 
\begin{equation}\label{E1}
\begin{tikzcd}
H^1(X,\cO_X)\ar[r,"\sigma"]\ar[d] & H^1(X,L_X)\ar[d]\\
H^1(\widetilde{X},\cO_{\widetilde{X}})\ar[r] & H^1(\widetilde{X},\Omega^1_{\widetilde{X}})
\end{tikzcd} \end{equation}
Since \(H^1(\widetilde{X}, \cO_{\widetilde{X}}) = 0\), \(\sigma = 0\) if the right vertical map \(H^1(X, L_X) \to H^1(\widetilde{X}, \Omega^1_{\widetilde{X}})\) is an isomorphism.

To prove this, we claim that there exists a short exact sequence of sheaves 
\begin{equation}\label{ses comparing cotangent sheaves}
	0\to \cO_P \to L_X \to \pi_*\Omega^1_{\widetilde{X}}\to 0,
\end{equation}
here $\cO_P$ denotes the skyscraper sheaf at the node $P$. The map $L_X\to \pi_*\Omega^1_{\widetilde{X}}$ is an isomorphism away from the node $P$, so we verify it locally.

On an affine open neighborhood $U\ni P$ with $\cO_X(U)=\BC[x,y]/(xy)$, the cotangent sheaf satisfies:
\[
L_X(U) = T\langle dx, dy \rangle/(x\,dy + y\,dx = 0), \quad T = \mathbb{C}[x,y]/(xy).
\]

It is easy to see that $\pi^{-1}(U)$ is a disjoint of two lines, we use $s, t$ to denote their coordinates. Then $\Omega^1_{\widetilde{X}}(\pi^{-1}(U))=\BC[s]ds\oplus \BC[t]dt$. 

Recall that on $U$ the map $\cO_X\to \pi_*\cO_{\widetilde{X}}$ is given by the ring map
\[\BC[x,y]/(xy)\to \BC[s]\oplus \BC[t] , \quad x\mapsto (s,0), y\mapsto (0,t). \]
So on $U$ the map $L_X \to \pi_*\Omega^1_{\widetilde{X}}$ is given by
\begin{equation}\label{map between cotangent sheaves}
	T\langle dx,dy \rangle/(xdy+ydx=0) \to \BC[s]ds\oplus \BC[t]dt, \quad dx\mapsto (ds,0), dy \mapsto (0,dt).
\end{equation} 
Notice that this is map of $\cO_X(U)$-module, where the module structure on $\pi_*\Omega^1_{\widetilde{X}}(U)$ is given by the action 
\[(f_1(x)+f_2(y))\cdot (g(s)ds,h(t)dt) := (f_1(s)g(s)ds,f_2(t)h(t)dt). \]
This makes sense since in $\cO_X(U)$, $xy=0$, so any element could be represented by $f_1(x)+f_2(y)$.

Now it is easy to see that the map \ref{map between cotangent sheaves} is surjective, with the kernel generated by the element $xdy$. However, notice that $T\cdot xdy \cong \BC \cdot xdy \oplus \BC \cdot x^2dy \oplus \cdots \cong \BC \cdot xdy$, since $x^2dy=x(-ydx)=-xydx=0$ in $T\langle dx,dy\rangle /(xdy+ydx=0)$. Hence the kernel is a skyscraper sheaf $\cO_P$. This proves the short exact sequence \ref{ses comparing cotangent sheaves}.

This short exact sequence induces a long exact sequence on cohomology
\[\cdots \to H^1(\cO_P)\to H^1(L_X) \to H^1(\pi_*\Omega^1_{\widetilde{X}}) \to H^2(\cO_P) \to \dots.  \]
Since $H^1(\cO_P)=H^2(\cO_P)=0$, \(H^1(L_X) \to H^1(\pi_*\Omega^1_{\widetilde{X}})\) is an isomorphism, proving \(\sigma = 0\).

\end{proof}

Combining Lemmas \ref{alpha}, \ref{beta}, \ref{gamma} and \ref{sigma} we conclude the proof of Theorem \ref{HdR}. The first page of the Hodge-to-de Rham spectral sequence is explicitly given by:
\begin{equation}\label{First page of HdR}
\begin{tikzcd}[row sep=tiny,column sep=small]
|[draw]|H^1(\cO_X)\ar[r,"0"] & |[draw]| H^1(\BLb)\\
|[draw]|H^0(\cO_X)\ar[r,"0"] & H^0(\BLb)\ar[r,"\simeq"] & H^0(\bigwedge^2\BLb) \\
& & H^{-1}(\bigwedge^2 \BLb)\ar[r,"\simeq"] & H^{-1}(\bigwedge^3\BLb)\\
& & & H^{-2}(\bigwedge^3\BLb)\ar[r,"\simeq"] & H^{-2}(\bigwedge^4\BLb)\\
& & & & \ddots & \ddots
\end{tikzcd} 
\end{equation}
The boxed terms $H^1(\cO_X)$, $H^1(\BLb_X)$ and $H^0(\cO_X)$ remain unchanged until the $\prescript{\infty}{}{E}$ page.

\section{Negative cyclic homology}\label{Negative cyclic homology}
With Theorem \ref{HdR}, we can compute the negative cyclic homology $\HN_*(X)$.

In \cite{Antieau}, Antieau establishes a decreasing filtration on \(\HN_*(X)\) with graded pieces
\[
\mathrm{gr}^n \HN_{-*}(X) \cong \rR^{*}\Gamma\left(X, \bF_H^n \widehat{\dR}^{\bullet}[2n]\right),
\]  
where \(\bF_H^{\bullet} \widehat{\dR}^{\bullet}\) denotes the Hodge filtration (stupid filtration) of the derived de Rham complex. The computation proceeds in two steps:
\begin{enumerate}
    \item \textbf{Case \(n = 0\):}  
    Here, \(\bF_H^0\widehat{\dR}^{\bullet}[0] \cong \widehat{\dR}^{\bullet}\). By \cite{Bhatt},  
    \[
    \mathrm{gr}^0 \HN_{-*}(X) = \rR^{*}\Gamma(\widehat{\dR}^{\bullet}) \cong H^{*}(X, \mathbb{C}).
    \]  
    This result extends to \(n < 0\) with a degree shift by \(2n\).
    \item \textbf{Case \(n = 1\):}  
    The truncated complex \(\bF_H^1\widehat{\dR}^{\bullet}[2]\) induces a spectral sequence whose \(\prescript{1}{}E\)-page matches columns \(\geq 1\) of the Hodge-to-de Rham spectral sequence in \eqref{First page of HdR}, with identical differentials. The only nontrivial cohomology appears in degree 2, giving  
    \[
    \mathrm{gr}^1 \HN_{-*}(X) = \begin{cases} 
    \mathbb{C} & \text{in degree } 0, \\
    0 & \text{otherwise}.
    \end{cases}
    \]  
    For \(n \geq 2\), \(\bF_H^n\widehat{\dR}^{\bullet}\) similarly concentrates cohomology in degree 2, adjusted by the \(2n\)-degree shift.
\end{enumerate}
To summarize, we have the following chart of dimensions of vector spaces
\begin{table}[H]
\begin{tabular}{c|cccccccccccc}
                          * & -4 & -3 & -2 & -1 & 0 & 1 & 2 & 3 & 4 & 5 & 6 \\
                           \hline
$gr^{-2}\HN_{-*}$ &    &    &    &    &   &   &   &   & 1 & 1 & 1 \\
$gr^{-1}\HN_{-*}$ &    &    &    &    &   &   & 1 & 1 & 1 &   &   \\
$gr^0\HN_{-*}$      &    &    &    &    & 1 & 1 & 1 &   &   &   &   \\
$gr^1\HN_{-*}$      &    &    &    &    & 1  &  &   &   &   &   &   \\
$gr^2\HN_{-*}$      &    &    & 1  &    &   &   &   &   &   &   &  \\
$gr^3\HN_{-*}$      &  1  &    &   &    &   &   &   &   &   &   &  
\end{tabular}
\end{table} 

Here the first row $*$ is the cohomological degree. After switch to homological degree, we conclude that
\begin{theorem}
The negative cyclic homology for projective nodal cubic curve $X$ is given by
\[\HN_n(X)=\begin{cases}
\BC^2 & \text{if}\ n\leq 0 \text{ and even}\\
\BC & \text{if}\ n\leq 0 \text{ and odd}\\
\BC & \text{if}\ n> 0 \text{ and even}\\
0 & \text{otherwise.}
\end{cases} \]
\end{theorem}

\section{General cases}\label{General cases}
In this section, we study the general case. Let \(Y\) be a projective curve of genus \(g\) with \(n\) singular points \(P_1, \ldots, P_n\), and let \(\pi : \widetilde{Y} \to Y\) denote its normalization. One easily verifies that \(\widetilde{Y}\) is a smooth projective curve of genus \(g-n\). Denote by \(L_Y\) the cotangent sheaf of \(Y\) and by \(\BLb_Y\) its cotangent complex. We then have the following result.
\begin{lemma}
The cohomology of $L_Y$ is given by
\[H^i(Y,L_Y)=\begin{cases}
\BC^g & i=0, \\
\BC & i=1, \\
0 & \mathrm{otherwise}.
\end{cases}.\] 
\end{lemma}
\begin{proof}
	We claim that $L_Y$ fits into the short exact sequence
\[0\to \bigoplus_{i=1}^n \cO_{P_i} \to L_Y \to \pi_*\Omega_{\widetilde{Y}}\to 0, \]
where $\cO_{P_i}$ denotes the skyscraper sheaf at $P_i$, $\Omega_{\widetilde{Y}}$ is the cotangent bundle on $\widetilde{Y}$. The proof follows from a local calculation similar to that in Lemma \ref{sigma}.

This short exact sequence induces a long exact sequence in cohomology:
\[0\to H^0\Bigl(\bigoplus_{i=1}^n \cO_{P_i}\Bigr) \to H^0(L_Y) \to H^0(\widetilde{Y},\Omega_{\widetilde{Y}})\to 0\to H^1(L_Y)\to H^1(\widetilde{Y},\Omega_{\widetilde{Y}})\to 0, \]
since \(H^1\Bigl(\bigoplus_{i=1}^n \cO_{P_i}\Bigr)=0\). By Serre duality, 
\[
H^1(\widetilde{Y},\Omega_{\widetilde{Y}})=H^0(\widetilde{Y},\cO_{\widetilde{Y}})=\BC,\quad H^0(\widetilde{Y},\Omega_{\widetilde{Y}})=\BC^{g-n},
\]
substituting these into the long exact sequence yields \(H^0(L_Y)=\BC^g\) and \(H^1(L_Y)=\BC\).
\end{proof}

Next, we study the higher derived wedge powers of the cotangent complex \(\BLb_Y\). Notice that for \(p \geq 2\), the sheaf \(\bigwedge^p \BLb_Y\) is supported at the nodal singularities \(P_i\) and can be computed locally. Hence, we have
\[
\bigwedge^p \BLb_Y \cong \bigoplus_{i=1}^n \bigwedge^p\BLb_Y\vert_{Z_i} \cong \bigoplus_{i=1}^n \bigwedge^p \BLb_X,
\]
where \(Z_i\) is a formal neighborhood of \(P_i\). This idea is also used in the proof of Lemma \ref{gamma}.

\begin{lemma}
    The hypercohomology of \(\bigwedge^k\BLb_X\) is given by
    \[
    \BH^i\Bigl(\bigwedge^k\BLb_X\Bigr)=
    \begin{cases}
    \BC^n, & \text{if } i = -k+2 \text{ or } i = -k+1, \\
    0, & \text{otherwise}.
    \end{cases}
    \]
\end{lemma}

\begin{proof}
    This follows directly from the above claim and Corollary \ref{hypercohomology}.
\end{proof}

We now study the Hodge-to-de Rham spectral sequence for \(Y\).

\begin{theorem}
    The Hodge-to-de Rham spectral sequence for \(Y\) degenerates at the \(\prescript{2}{}E\)-page.
\end{theorem}

\begin{proof}
    We begin by writing a few terms of the \(\prescript{1}{}E\)-page for \(Y\):
    \begin{equation}\label{HdR spectral seq for Y}
        \begin{tikzcd}[row sep=tiny,column sep=small]
            1 & H^1(\cO_Y)\ar[r,"\sigma"] & H^1(L_Y)\\[1mm]
            0 & H^0(\cO_Y)\ar[r,"\alpha"] & H^0(L_Y)\ar[r,"\gamma"] & H^0\bigl(\bigwedge^2\BLb_Y\bigr) \\[1mm]
            -1 & & & H^{-1}\bigl(\bigwedge^2 \BLb_Y\bigr)\ar[r,"\beta_1"] & H^{-1}\bigl(\bigwedge^3\BLb_Y\bigr)\\[1mm]
            -2 & & & & H^{-2}\bigl(\bigwedge^3\BLb_Y\bigr)\ar[r,"\beta_2"] & H^{-2}\bigl(\bigwedge^4\BLb_Y\bigr)\\[1mm]
            \vdots& & & & & \ddots & \ddots
        \end{tikzcd} 
    \end{equation}
    
    Our proof proceeds in four steps:
    \begin{itemize}
        \item[Step 1:] \(\alpha=0\).
        \item[Step 2:] \(\beta_k\) is an isomorphism for all \(k\geq 1\).
        \item[Step 3:] \(\gamma\) is surjective.
        \item[Step 4:] \(\sigma=0\).
    \end{itemize}
    
    \begin{enumerate}
        \item \textbf{Step 1:} The spectral sequence converges at the \(\prescript{\infty}{}E\)-page to the singular cohomology \(H^*(Y)\), where
        \[
        H^i(Y)=
        \begin{cases}
        \BC & i=0,2,\\[1mm]
        \BC^{2g-n} & i=1,\\[1mm]
        0 & \text{otherwise}.
        \end{cases}
        \]
        Since the only term on the 0-th diagonal must survive, no nonzero differential can originate from \(H^0(\cO_Y)\), so \(\alpha=0\).
        
        \item \textbf{Step 2:} This step follows directly from our computation of \(\bigwedge^p\BLb_Y\) and Lemma \ref{beta}.
        
        \item \textbf{Step 3:} A computation similar to that in Lemma \ref{gamma} shows that the following diagram
        \[
        \begin{tikzcd}
        H^0\Bigl(\bigoplus_{i=1}^n\cO_{P_i}\Bigr)\ar[r]\ar[d,equal] & H^0(L_Y)\ar[r,"\gamma"]\ar[d] & H^0\Bigl(\bigwedge^2\BLb_Y\Bigr)\ar[d,"\simeq"{anchor=south,rotate=-90}]\\[1mm]
        H^0\Bigl(\bigoplus_{i=1}^n\cO_{P_i}\Bigr)\ar[r]\ar[d,equal] & H^0\Bigl(L_Y\vert_{U}\Bigr)\ar[r]\ar[d,"\simeq"{anchor=south,rotate=-90}] & H^0\Bigl(\bigwedge^2 \BLb_Y\vert_{U}\Bigr)\ar[d,"\simeq"{anchor=south,rotate=-90}]\\[1mm]
        H^0\Bigl(\bigoplus_{i=1}^n\cO_{P_i}\Bigr)\ar[r]\ar[d,equal] & H^0\Bigl(L_Z\vert_{V}\Bigr)\ar[r] & H^0\Bigl(\bigwedge^2 \BLb_Z\vert_{V}\Bigr)\\[1mm]
        H^0\Bigl(\bigoplus_{i=1}^n\cO_{P_i}\Bigr)\ar[r] & H^0(L_Z)\ar[r]\ar[u] & H^0\Bigl(\bigwedge^2 \BLb_Z\Bigr)\ar[u,"\simeq"{anchor=south,rotate=-90}]
        \end{tikzcd}
        \]
        commutes, where \(Z\) is a disjoint union of \(n\) copies of the affine scheme \(\Spec \BC[x,y]/(xy)\), and 
        \[
        U = \bigsqcup_{i=1}^n U_i,\quad V = \bigsqcup_{i=1}^n V_i
        \]
        are disjoint unions of formal open neighborhoods of nodes \(P_i\) in \(X\) and \(Z\), respectively. Since \(Z\) consists of \(n\) copies of a crossing of two lines, each irreducible component can be computed as in Lemma \ref{gamma}. Hence, the final row is surjective, and so \(\gamma: H^0(L_Y) \to H^0\bigl(\bigwedge^2\BLb_Y\bigr)\) is surjective.
        
        \item \textbf{Step 4:} Consider the following diagram displaying the ranks of the groups in the first two rows:
        \[
        \begin{tikzcd}[row sep=tiny,column sep=small]
        1 & \BC^g \ar[r,"\sigma"] & \BC\\[1mm]
        0 & \BC \ar[r,"\alpha=0"] & \BC^g \ar[r,twoheadrightarrow,"\gamma"] & \BC^n 
        \end{tikzcd}
        \]
        All other terms vanish on the \(\prescript{2}{}E\)-page, so the only nonzero term on the second diagonal is \(H^1(L_Y) \cong \BC\). Since \(H^2(Y)=\BC\), it follows that on the final page there is exactly one nonzero term in the second diagonal. Consequently, \(\sigma\) must be the zero map.
    \end{enumerate}
    
    Combining these four steps, we conclude that the Hodge-to-de Rham spectral sequence for any nodal curve \(Y\) degenerates at the \(\prescript{2}{}E\)-page. In particular, the \(\prescript{2}{}E\)-page (also the $\prescript{\infty}{}E$-page) is given by
    \[
    \begin{tikzcd}[row sep=tiny,column sep=small]
    1 & \BC^g & \BC\\[1mm]
    0 & \BC & \BC^{g-n} & 0 \\[1mm]
    -1 & & & 0  & 0\\[1mm]
    -2 & & & & 0 & 0\\[1mm]
    \vdots& & & & & \ddots & \ddots
    \end{tikzcd}
    \]
\end{proof}

We also compute the Hochschild homology and negative cyclic homology for \(Y\).

\begin{theorem}
    The Hochschild homology of \(Y\) is 
    \[
    \HH_i(Y)=
    \begin{cases}
        \BC^g & \text{if } i=-1,1,\\[1mm]
        \BC^2 & \text{if } i=0,\\[1mm]
        \BC^n & \text{if } i\geq 2,\\[1mm]
        0 & \text{otherwise},
    \end{cases}
    \]
    and the negative cyclic homology of \(Y\) is
    \[
    \HN_i(Y)=
    \begin{cases}
        \BC^2 & \text{if } i\leq 0 \text{ and even},\\[1mm]
        \BC^{2g-n} & \text{if } i<0 \text{ and odd},\\[1mm]
        \BC^{g-n} & \text{if } i=1,\\[1mm]
        \BC^n & \text{if } i>0 \text{ and even},\\[1mm]
        0 & \text{otherwise}.
    \end{cases}
    \]
\end{theorem}

\begin{proof}
    The computation of \(\HH_*(Y)\) follows directly from the HKR isomorphism \cite{Buchweitz-Flenner}:
    \[
    \HH_i(Y) \simeq \prod_{q-p=i} H^p\Bigl(Y,\bigwedge^q\BLb_Y\Bigr).
    \]
    To compute \(\HN_*(Y)\), we mimic the discussion in Section \ref{Negative cyclic homology} by studying the decreasing filtration on \(\HN_*(Y)\). The graded pieces are given by
    \[
    \mathrm{gr}^k \HN_{-*}(Y) \cong \rR^*\Gamma\Bigl(Y,\bF_H^k \widehat{\dR}^{\bullet}[2k]\Bigr),
    \]
    where \(\bF_H^{\bullet}\widehat{\dR}^{\bullet}\) denotes the Hodge (stupid) filtration of the derived de Rham complex \(\widehat{\dR}^{\bullet}\). We now compute \(\HN_*(Y)\) in three steps:
    \begin{enumerate}
        \item \textbf{Step 1:} \(k=0\). Since \(\bF_H^0\widehat{\dR}^{\bullet}[0]\cong \widehat{\dR}^{\bullet}\), it follows from \cite{Bhatt} that
        \[
        \mathrm{gr}^0\HN_{-*}(Y)=\rR^*\Gamma(\widehat{\dR}^{\bullet})\cong H^*(Y,\BC).
        \]
        This argument generalizes to any \(k<0\) with a shift in degree by \(2k\).
        
        \item \textbf{Step 2:} \(k=1\). In this case, we compute the hypercohomology of the truncated complex \(\bF_H^1(\widehat{\dR}^{\bullet}[2])\). The Hodge filtration on \(\bF_H^1(\widehat{\dR}^{\bullet})\) induces a spectral sequence whose first page is identical to the first page of the Hodge-to-de Rham spectral sequence \eqref{HdR spectral seq for Y} for \(Y\), but with only columns \(\geq 1\). Moreover, the differentials in this spectral sequence are the same as those in \eqref{HdR spectral seq for Y}. It is then straightforward to verify that the only nontrivial cohomology groups are \(\BC^{g-n}\) in degree 1 and \(\BC\) in degree 2. Therefore,
        \[
        \mathrm{gr}^1\HN_{-*}(Y)=
        \begin{cases}
            \BC & \text{in degree } 0,\\[1mm]
            \BC^{g-n} & \text{in degree -1},\\[1mm]
            0 & \text{in all other degrees}.
        \end{cases}
        \]
        
        \item \textbf{Step 3:} \(k=2\). Here, we examine the spectral sequence \eqref{HdR spectral seq for Y} considering only the columns \(\geq 2\). One finds that the only nontrivial cohomology is \(\BC^n\) in degree 2, so
        \[
        \mathrm{gr}^2\HN_{-*}(Y)=
        \begin{cases}
            \BC^n & \text{in degree -2},\\[1mm]
            0 & \text{in all other degrees}.
        \end{cases}
        \]
        This argument generalizes to any \(k\geq 3\) since each \(\bF_H^k\widehat{\dR}^{\bullet}\) has nontrivial cohomology only in degree 2, with an appropriate \(2k\) degree shift.
    \end{enumerate}
    To summarize, we obtain the following chart of the dimensions of the graded pieces:
    \begin{table}[H]
    \centering
    \begin{tabular}{c|cccccccccccc}
        * & -4 & -3 & -2 & -1 & 0 & 1 & 2 & 3 & 4 & 5 & 6 \\ \hline
        \(gr^{-2}\HN_{-*}\) &    &    &    &    &   &   &   &   & 1 & \(2g-n\) & 1 \\
        \(gr^{-1}\HN_{-*}\) &    &    &    &    &   &   & 1 & \(2g-n\) & 1 &   &   \\
        \(gr^{0}\HN_{-*}\)  &    &    &    &    & 1 & \(2g-n\) & 1 &   &   &   &   \\
        \(gr^{1}\HN_{-*}\)  &    &    &    & \(g-n\)  & 1  &   &   &   &   &   &   \\
        \(gr^{2}\HN_{-*}\)  &    &    & \(n\)  &    &   &   &   &   &   &   &  \\
        \(gr^{3}\HN_{-*}\)  & \(n\)  &    &   &    &   &   &   &   &   &   &  
    \end{tabular}
    \end{table}
    
    The first row (labeled \( * \)) indicates the cohomological degree. After switching to homological degree, we obtain the stated result.
\end{proof}

\section{Liftable classes}\label{Liftable classes}

In this section, we study the natural map \(\HN_* \to \HH_*\). Before doing so, we introduce the \emph{Hochschild-to-cyclic spectral sequence}. Recall that for any variety \(W\), there exists a spectral sequence whose first page is given by
\[
\begin{tikzcd}[row sep=small]
& \vdots & \vdots & \vdots \\
\cdots & 0 \ar[r] & HH_{-1}\ar[r,"uB"] & uHH_0\ar[r,"uB"] & \cdots \\
\cdots & 0 \ar[r] & HH_0\ar[r,"uB"] & uHH_1\ar[r,"uB"] & \cdots \\
\cdots & 0 \ar[r] & HH_1\ar[r,"uB"] & uHH_2\ar[r,"uB"] & \cdots \\
& \vdots & \vdots & \vdots
\end{tikzcd}
\]
where \(u\) is a formal variable of homological degree \(-2\) and \(B\) is Connes's operator. This spectral sequence converges at the \(\prescript{\infty}{}E\)-page to the negative cyclic homology \(\HN_*(W)\).

\begin{theorem}
    The Hochschild-to-cyclic spectral sequence for the nodal cubic curve \(X\) degenerates at the \(\prescript{2}{}E\)-page.
\end{theorem}

\begin{proof}
    The proof is straightforward once we apply the Hochschild–Kostant–Rosenberg (HKR) isomorphism to the terms on the first page. Recall that by the HKR isomorphism \cite{Buchweitz-Flenner},
    \[
    \HH_k(X) \simeq \prod_{q-p=k} H^p\Bigl(X,\bigwedge^q\BLb_X\Bigr).
    \]
    Moreover, the map
    \[
    uB:\prod_{q-p=k}H^p\Bigl(X,\bigwedge^q\BLb_X\Bigr)\to \prod_{q-p=k}H^p\Bigl(X,\bigwedge^{q+1}\BLb_X\Bigr)
    \]
    is induced by the de Rham differential \(d:\bigwedge^q\BLb_X\to \bigwedge^{q+1}\BLb_X\); that is, it is the direct product of the maps
    \[
    H^p\Bigl(X,\bigwedge^q\BLb_X\Bigr)\to H^p\Bigl(X,\bigwedge^{q+1}\BLb_X\Bigr).
    \]
    Thus, we can rewrite the first page as
    \[
    \begin{tikzcd}[row sep=small,column sep=small]
        & \vdots & \vdots & \vdots \\
        \cdots & 0 \ar[r] & H^1(\cO) \ar[r,"0"] & u\left(\substack{H^0(\cO)\\ \oplus \\H^1(L)}\right) \ar[r,"0"] & u^2H^0(L) \ar[r,"\simeq"] & \cdots \\[1mm]
        \cdots & 0 \ar[r] & \substack{H^0(\cO)\\ \oplus \\H^1(L)} \ar[r,"0"] & uH^0(L) \ar[r,"\simeq"] & u^2H^0(\bigwedge^2\BLb_X) \ar[r,"0"] & \cdots \\[1mm]
        \cdots & 0 \ar[r] & H^0(L) \ar[r,"\simeq"] & uH^0(\bigwedge^2\BLb) \ar[r,"0"] & u^2H^{-1}(\bigwedge^2\BLb) \ar[r,"\simeq"] & \cdots \\[1mm]
        \cdots & 0 \ar[r] & H^0(\bigwedge^2\BLb) \ar[r,"0"] & uH^{-1}(\bigwedge^2\BLb) \ar[r,"\simeq"] & u^2H^{-1}(\bigwedge^3\BLb) \ar[r,"0"] & \cdots \\
        & \vdots & \vdots & \vdots
    \end{tikzcd}
    \]
    where the differentials are determined by our study of the HdR spectral sequence. Consequently, the \(\prescript{2}{}E\)-page takes the form
    \[
    \begin{tikzcd}[row sep=small]
        & \vdots & \vdots & \vdots \\
        \cdots & 0  & H^1(\cO)\ar[rrd] & u\left(\substack{H^0(\cO)\\ \oplus \\H^1(L)}\right) & 0 & \cdots \\[1mm]
        \cdots & 0  & \substack{H^0(\cO)\\ \oplus \\H^1(L)}\ar[rrd] & 0 & 0 & \cdots \\[1mm]
        \cdots & 0  & 0 & 0 & 0 & \cdots\\[1mm]
        \cdots & 0  & H^0\Bigl(\bigwedge^2\BLb\Bigr) & 0 & 0 & \cdots \\
        & \vdots & \vdots & \vdots
    \end{tikzcd}
    \]
    It is clear from this description that the spectral sequence degenerates at the \(\prescript{2}{}E\)-page.
\end{proof}

\begin{corollary}\label{HNtoHH}
    The natural map \(\HN_n(X)\to \HH_n(X)\) is 
    \begin{itemize}
        \item an isomorphism if \(n=-1\) or if \(n\geq 0\) is even,
        \item the zero map otherwise.
    \end{itemize}
\end{corollary}

\begin{proof}
    This follows directly from the spectral sequence described above. Since the \(\prescript{2}{}E\)-page coincides with the \(\prescript{\infty}{}E\)-page, the \(n\)-th diagonal yields a filtration of \(\HN_n(X)\) that either is trivial or contains only one nonzero graded piece. In the latter case, we obtain \(\HN_n(X)\cong \HH_n(X)\).
\end{proof}

Similarly, we study the general nodal curve \(Y\).

\begin{theorem}
    The Hochschild-to-cyclic spectral sequence for \(Y\) degenerates at the \(\prescript{2}{}E\)-page.
\end{theorem}

\begin{proof}
    The proof is again straightforward after applying the HKR isomorphism. The first page may be rewritten as
    \[
    \begin{tikzcd}[row sep=small]
        & \vdots & \vdots & \vdots \\
        \cdots & 0 \ar[r] & H^1(\cO) \ar[r,"0"] & u\left(\substack{H^0(\cO)\\ \oplus \\H^1(L)}\right) \ar[r,"0"] & u^2H^0(L) \ar[r,twoheadrightarrow,"u\gamma"] & \cdots \\[1mm]
        \cdots & 0 \ar[r] & \substack{H^0(\cO)\\ \oplus \\H^1(L)} \ar[r,"0"] & uH^0(L) \ar[r,twoheadrightarrow,"u\gamma"] & u^2H^0(\bigwedge^2\BLb_X) \ar[r,"0"] & \cdots \\[1mm]
        \cdots & 0 \ar[r] & H^0(L) \ar[r,twoheadrightarrow,"u\gamma"] & uH^0(\bigwedge^2\BLb) \ar[r,"0"] & u^2H^{-1}(\bigwedge^2\BLb) \ar[r,"\simeq"] & \cdots \\[1mm]
        \cdots & 0 \ar[r] & H^0(\bigwedge^2\BLb) \ar[r,"0"] & uH^{-1}(\bigwedge^2\BLb) \ar[r,"\simeq"] & u^2H^{-1}(\bigwedge^3\BLb) \ar[r,"0"] & \cdots \\
        & \vdots & \vdots & \vdots
    \end{tikzcd}
    \]
    where the maps labeled \(u\gamma\) denote the induced maps from the de Rham differential. Hence, the \(\prescript{2}{}E\)-page is given by
    \[
    \begin{tikzcd}[row sep=small,column sep=tiny]
        && \vdots & \vdots & \vdots & \vdots & \vdots \\
        2 & \cdots & 0 & 0  & H^1(\cO)\ar[rrd] & u\left(\substack{H^0(\cO)\\ \oplus \\H^1(L)}\right) & u^2\ker(u\gamma)  & \cdots \\[1mm]
        1 & \cdots & 0  & H^1(\cO)\ar[rrd] & u\left(\substack{H^0(\cO)\\ \oplus \\H^1(L)}\right) & u^2\ker(u\gamma) & 0 & \cdots \\[1mm]
        0 & \cdots & 0  & \substack{H^0(\cO)\\ \oplus \\H^1(L)}\ar[rrd] & u\ker(u\gamma) & 0 & 0 & \cdots \\[1mm]
        -1 & \cdots & 0  & \ker(u\gamma) & 0 & 0 & 0 & \cdots \\[1mm]
        -2 & \cdots & 0  & H^0\Bigl(\bigwedge^2\BLb\Bigr) & 0 & 0 & 0 & \cdots \\
        && \vdots & \vdots & \vdots & \vdots & \vdots
    \end{tikzcd}
    \]
    By inspecting the directions of the differential maps, one easily verifies that there can be no further nonzero differentials, so the spectral sequence degenerates at the \(\prescript{2}{}E\)-page.
\end{proof}

Classifying all liftable Hochschild classes is crucial for computing categorical enumerative invariants (CEI) \cite{Caldararu-Tu_24}. Roughly speaking, CEI are invariants associated to an \(\cA_{\infty}\)-algebra \(A\) together with additional data. Given Hochschild classes in \(\HH_*(A)\) as input, a CEI computation produces complex numbers. Originally, these computations were defined only for smooth and proper \(\cA_{\infty}\)-algebras, so one cannot compute the CEI of the nodal cubic curve. However, C\u{a}ld\u{a}raru and Tu conjecture that for a nonsmooth \(\cA_{\infty}\)-algebra \(A\), such computations can still be carried out provided that both the inserted classes and the intermediate classes are liftable to \(\HN_*(A)\).

In particular, Corollary \ref{HNtoHH} implies that we may be able to compute the CEI for the nodal cubic curve using insertion classes in \(\HH_{-1}(X)\). Combined with our observation that CEIs satisfy a holomorphic anomaly equation, this reduction allows us to compute genus \(\leq 5\) CEIs for any elliptic curve by reducing the problem to computing genus \(\leq 5\) CEIs for the special nodal cubic curve, which is more approachable numerically.

\begin{remark}
    We have attempted to apply the same method to study the degenerate quintic 
    \[
    x_0x_1\cdots x_4=0
    \]
    in \(\BP^4\), which is also of interest for CEI computations. However, its Hodge-to-de Rham spectral sequence does not degenerate at the \(\prescript{2}{}E\)-page.
\end{remark}

\section{Appendix: cuspidal curve}\label{Appendix: cuspidal curve}
Our study of nodal cubic curve has a strong motivation from enumerative geometry, but we can apply the same ideas to study the projective cuspidal cubic curve $C$. The proof will be easier than the nodal curve case. We just outline some of the results we get, and sketch the proofs.

\begin{theorem}
For the cuspidal cubic curve $C$, 
\begin{enumerate}
\item its Hochschild homology is given by
\[\HH_n(C)=\begin{cases}
\BC & n=-1\\
\BC^3 & n=0\\
\BC^2 & n> 0
\end{cases} \]
\item its Hodge to de Rham spectral sequence degenerates at page $\prescript{2}{}E$;
\item its negative cyclic homology is given by 
\[\HN_n(C)=\begin{cases}
\BC^3 & n=0\\
\BC^2 & n\neq 0\text{ and even}\\
0 & {otherwise} 
\end{cases} \]
\item The natural map $\HN_n(C)\to \HH_n(C)$ is 
\begin{itemize}
\item an isomorphism, if $n\geq 0$ and even,
\item 0, otherwise.
\end{itemize}
\end{enumerate}

\end{theorem}

\begin{proof}
As before the cotangent sheaf $L_C$ also admits a resolution 
\[0\to \cO_C(-3)\to \Omega_{\BP^2}^1\vert_C \to L_C\to 0. \]
We can still compute the derived exterior powers of $\BL_C^{\bullet}$, for example
\[\bigwedge^2 \BL_C^{\bullet} = \cO_C(-3)\otimes \left(0\to\cO_C(-3)\to\Omega_{\BP^2}^1\vert_C\to \cO_C \to 0 \right)[0], \]
and this will be a local calculation since $\bigwedge^2\BL_C^{\bullet}$ supports at the singular point. Using the affine local model, we can compute the cohomology of the above chain complex $H^0(\bigwedge^2 \BL_C^{\bullet})=H^{-1}(\bigwedge^2\BL_C^{\bullet})=\BC^2$. Then the remaining computations are similar to the computation we have done for the nodal cubic curve $X$.
\end{proof}

\printbibliography

\end{document}